\documentclass[plain,12pt]{article}
\usepackage{amsmath}
\usepackage{amssymb}
\usepackage{amsthm}
\usepackage{mathrsfs}
\usepackage[all]{xy}
\usepackage{pstricks,pst-node}
\newtheorem{theorem}{Theorem}[section]
\newtheorem{lemma}[theorem]{Lemma}

\theoremstyle{remark}

\theoremstyle{definition}

\DeclareMathOperator{\core}{core}

\author{Neil Saunders}
\title{A Strict Inequality for a Minimal Degree of a Direct Product}

\begin{document}

\maketitle

\begin{abstract}
The minimal faithful permutation degree $\mu(G)$ of a finite group $G$ is the least non-negative integer $n$ such that $G$ embeds in the
symmetric group $Sym(n)$. Work of Johnson and Wright in the 1970's established conditions for when $\mu(H \times K) = \mu(H) + \mu(K)$, for
finite groups $H$ and $K$. Wright asked whether this is true for all finite groups. A counter-example of degree 15 was provided by the referee
and was added as an addendum in Wright's paper. Here we provide a counter-example of degree 12.
\end{abstract}

\section{Introduction}
The minimal faithful permutation degree $\mu(G)$ of a finite group $G$ is the least non-negative integer $n$ such that $G$ embeds in the
symmetric group $Sym(n)$. It is well known that $\mu(G)$ is the smallest value of $\sum_{i=1}^{n}|G:G_i|$ for a collection of subgroups
$\{G_1,\ldots, G_n\}$ satisfying $\bigcap_{i=1}^{n}\core(G_i)=\{1\}$, where $\core(G_i)=\bigcap_{g \in G}G_{i}^{g}$. \vspace{12pt}

We first give a theorem due to Karpilovsky \cite{K70} which will be needed later. The proof of it can be found in \cite{J71} or \cite{S05}.
\begin{theorem} \label{theorem:abelian}
Let A be a non-trivial finite abelian group and let $A \cong A_1 \times \ldots \times A_n$ be its direct product decomposition into non-trivial
cyclic groups of prime power order. Then $$\mu(A)= a_1 + \ldots + a_n,$$ where $|A_i|=a_i$ for each $i$.
\end{theorem}

One of the themes of Johnson and Wright's work was to establish conditions for when \begin{equation} \mu(H \times K)=\mu(H)+\mu(K)
\label{eq:directsum} \end{equation} for finite groups $H$ and $K$. The next result is due to Wright \cite{W75}.

\begin{theorem} \label{theorem:nilpotent}
Let $G$ and $H$ be non-trivial nilpotent groups. Then $\mu(G \times H)=\mu(G) + \mu(H)$.
\end{theorem}

Further in \cite{W75}, Wright constructed a class of groups $\mathscr{C}$ with the property that for all $G \in \mathscr{C}$, there exists a
nilpotent subgroup $G_1$ of $G$ such that $\mu(G_1)=\mu(G)$. It is a consequence of Thereom (\ref{theorem:nilpotent}) that $\mathscr{C}$ is
closed under direct products and so (\ref{eq:directsum}) holds for any two groups $H,K \in \mathscr{C}$. Wright proved that $\mathscr{C}$
contains all nilpotent, symmetric, alternating and dihedral groups, however the extent of it is still an open problem. In \cite{EP88}, Easdown
and Praeger showed that (\ref{eq:directsum}) holds for all finite simple groups. \vspace{12pt}

The counter-example to (\ref{eq:directsum}) was provided by the referee in Wright's paper \cite{W75} and involved subgroups of the standard
wreath product $C_5 \wr Sym(3)$, specifically the group $G(5,5,3)$ which is a member of a class of unitary reflection groups. We give a brief
exposition on these groups now. \vspace{12pt}

Let $m$ and $n$ be positive integers, let $C_m$ be the cyclic group of order $m$ and $B=C_m \times \ldots \times C_m$ be the product of $n$
copies of $C_m$. For each divisor $p$ of $m$ define the group $A(m,p,n)$ by $$A(m,p,n)= \{(\theta_1,\theta_2,\ldots, \theta_n) \in B  \ | \
(\theta_1\theta_2\ldots\theta_n)^{m/p}=1 \}.$$ It follows that $A(m,p,n)$ is a subgroup of index $p$ in $B$ and the symmetric group $Sym(n)$
acts naturally on $A(m,p,n)$ by permuting the coordinates. \vspace{12pt}

$G(m,p,n)$ is defined to be the semidirect product of  $A(m,p,n)$ by $Sym(n)$. It follows that $G(m,p,n)$ is a normal subgroup of index $p$ in
$C_m \wr Sym(n)$ and thus has order $m^nn!/p$. \vspace{12pt}

It is well known that these groups can be realized as finite subgroups of $GL_n(\mathbb{C})$, specifically as $n \times n$ matrices with exactly
one non-zero entry, which is a complex $m$th root of unity, in each row and column such that the product of the entries is a complex $(m/p)$th
root of unity. Thus the groups $G(m,p,n)$ are sometimes referred to as monomial reflection groups. For more details on the groups $G(m,p,n)$,
see \cite{OT92}.

\section{Calculation of $\mu(G(4,4,3))$}
Recall that $G(4,4,3)=A(4,4,3) \rtimes Sym(3)$, where $$A(4,4,3)=\{(\theta_1,\theta_2,\theta_3) \in C_4 \times C_4 \times C_4 \ | \
\theta_1\theta_2\theta_3=1 \}$$ which is isomorphic to a product of two copies of the cyclic group of order 4. Hence $$G(4,4,3) \cong (C_4
\times C_4) \rtimes Sym(3).$$ From now on, we will let $G$ denote $G(4,4,3)$. A presentation for this group can be given thus $$G= \langle
x,y,a,b | x^4=y^4=b^3=a^2=1, xy=yx, x^a=y, x^b=y, y^b=x^{-1}y^{-1}, b^a=b^{-1} \rangle.$$ Since $\langle x,y \rangle \cong C_4 \times C_4$ is a
proper subgroup of $G$ we have by Theorem \ref{theorem:abelian}, that $8 = \mu(\langle x,y \rangle)\leq \mu(G)$. Moreover since $G$ is a proper
subgroup of the wreath product $W:=C_4 \wr Sym(3)$, for which $\mu(W)=12$, we have the inequalities $$8 \leq \mu(G) \leq 12.$$ We will prove
that in fact $\mu(G)=12$ by a sequence of lemmas.

\begin{lemma} \label{lemma:tranisitive}
$\langle x^2, y^2 \rangle$ is the unique minimal normal subgroup of $G$.
\end{lemma}

\begin{proof}
Observe by the conjugation action of $a$ and $b$ on $x^2$ and $y^2$ that $M= \langle x^2, y^2 \rangle$ is indeed normal in $G$. Let $N$ be a
non-trivial normal subgroup of $G$ so there exists an $$\alpha=x^iy^jb^ka^l$$ in $N$ where $i,j \in \{0,1,2,3\},\ k \in \{0,1,2\},\ l \in
\{0,1\}$ are not all zero. It remains to show that $M$ is contained in $N$.\newline
\newline
\underline {\textbf{Case (a):}} $k=l=0$. \newline \underline{Subcase (i)}: $i=j$ so $\alpha=x^iy^i$.

Then $\alpha\alpha^b=x^iy^iy^ix^{-i}y^{-i}=y^i \in N$, so $y^{-i}\alpha=x^i \in N$. But $i \neq 0$, so $M \subseteq \langle x^i, y^i \rangle$.
Hence $M \subseteq N$, as required. \newline
\newline
\underline{Subcase (ii)}: $i+j \not\equiv 0$ mod $4$.

Then $\alpha\alpha^a=x^{i+j}y^{i+j}$ and we are back in Subcase (i).\newline
\newline
\underline{Subcase (iii)}: $i+j \equiv 0$ mod $4$.

Then $\alpha\alpha^b=x^{i-j}y^i$. If $2i-j \not\equiv 0 $ mod $4$, then we are back in Subcase (ii), so suppose $2i \equiv j$ mod $4$. Then
together with $i+j \equiv 0$ mod $4$ it follows that $i=0$. Therefore $j$ is zero and $\alpha$ is trivial. This completes case (a).\newline
\newline
\underline {\textbf{Case (b):}} $k \neq 0$ or $l\neq0$. \newline \underline{Subcase (i)}: $l=0$ so $k\neq 0$ \newline

Then $\alpha\alpha^{-b}=x^iy^jb^k(x^{-j}y^{i-j}b^k)^{-1}=x^{i+j}y^{2j-i}$. If $i+j \not\equiv 0$ or $2j-i \not\equiv 0 $ mod $4$, then we are
back in Case (a) so suppose  $i+j\equiv 2j-i \equiv 0$ mod $4$. Solving gives $i=j=0$ and so $\alpha=b^k$, whence $\langle b \rangle \in N$.
Hence $$b^{-1}b^x=b^{-1}x^{-1}bx=y^{-1}x \in N$$ and we are back in Case (a).\newline
\newline
\underline{Subcase (ii)}: $l\neq 0$ and $k\neq 0$.\newline

Then  $\alpha\alpha^{-a}= x^iy^jb^ka^l(x^{j}y^{i}b^{-k}a^l)^{-1}=x^iy^jb^ka^la^{-l}b^kx^{-j}y^{-i}=x^py^qb^{2k}$ where $p,q\in \{0,1,2,3\}$ and
we are back in Subcase (i), replacing $k$ by $2k$.\newline
\newline
\underline{Subcase (iii)}: $k=0$ so $l=1$ \newline Then $$\alpha\alpha^{-b}=x^iy^ja(x^iy^ja)^{-b}= x^py^qb^2$$ for some $p,q \in \{0,1,2,3\}$
and again we are back in Subcase (i). \vspace{12pt}

This completes the proof.
\end{proof}

It is worth observing at this point that Lemma \ref{lemma:tranisitive} tells us that any minimal faithful representation of $G$ is necessarily
transitive. That is, any minimal faithful collection of subgroups $\{G_1,\ldots,G_n\}$ is just a single core-free subroup.
\begin{lemma} \label{lemma:order}
Elements of $\langle x,y \rangle b$ and $\langle x,y \rangle b^2$ have order $3$. All other elements of $G$ have order dividing by $8$.
\end{lemma}

\begin{proof}
It is a routine calculation to show that any element of the form $\alpha=x^iy^jb^k$ for $k$ nonzero has order three. Now suppose
$\alpha=x^iy^jb^ka^l$ where $l$ is nonzero. Then $l=1$ and we have $$\alpha^2=x^py^q(b^ka)^2=x^py^q,$$ for some $p,q$, which has order dividing
4. Therefore $\alpha$ has order dividing 8.
\end{proof}

It is an immediate consequence that $G$ does not contain any element of order 6.

\begin{lemma}
If $L$ is a core-free subgroup of $G$ then $|G:L| \geq 12$.
\end{lemma}

\begin{proof}
Suppose for a contradiction that $\core(L)=\{1\}$ and $|G:L|< 12$. Since $|G|=96$, $|L| >8$. However, if $|L| >12$ then $|G:L| < 8$ and so
$\mu(G) <8$ contradicting that $\mu(G) \geq 8$. Therefore $|L|=12$ and so by the classification of groups of order 12, see \cite{P05}, $L$ is
isomorphic to one of the following groups

\begin{displaymath}
L \cong
\begin{cases}
C_{12} \\ C_6 \times C_2 \\ A_4 \\ D_6 \\ T=\langle s,t \ | \ s^6=1, s^3=t^2, sts=s \rangle
\end{cases}
\end{displaymath}

Notice that the groups $C_{12}, C_6 \times C_2, D_6$ and $T$ each contain an element of order $6$ and so cannot be isomorphic to $L$ by Lemma
\ref{lemma:order}. \vspace{12pt}

Hence $L$ is isomorphic to $A_4$ and so we can find two non-commuting elements $\alpha=x^iy^jb^k$ and $\beta=x^sy^tb^r$ of order three that
generate it such that $\alpha\beta$ has order two. Now $$\alpha\beta=x^py^qb^{k+r}$$ for some $p,q \in \{0,1,2,3\}$ and so $k+r \equiv 0$ mod
$3$ by Lemma \ref{lemma:order}. Without loss of generality let $k=1$. Now
\begin{displaymath}
\alpha\beta=
\begin{cases}
x^2\\ y^2\\ x^2y^2
\end{cases}
\end{displaymath}
and upon conjugation by $\alpha=x^iy^jb$, we get respectively,
\begin{displaymath}
(\alpha\beta)^{\alpha}=
\begin{cases}
y^2\\ x^2y^2 \\x^2.
\end{cases}
\end{displaymath}

So in each case we get $\langle x^2,y^2 \rangle \subseteq L$, contradicting that $L$ is core-free.
\end{proof}

Combining the above lemmas we find that any minimal faithful representation of $G$ is necessarily transitive and that any faithful transitive
representation has degree at least 12. Therefore we have $12 \leq \mu(G)$. But  $\mu(G) \leq 12$. Therefore we have proved the following:

\begin{theorem} \label{theorem:G443}
The minimal faithful permutation degree of $G(4,4,3)$ is $12$.
\end{theorem}

\section{G(4,4,3) forms a Counter-Example of Degree 12}

Let $W=C_4 \wr Sym(3)$ be the wreath product of the cyclic group of order 4 by the symmetric group on 3 letters. Observe at this point that
since the base group of $W$ is $C_4 \times C_4 \times C_4$, and $\mu(C_4 \times C_4 \times C_4)=12$ by Theorem \ref{theorem:abelian},
$\mu(W)=12$. Let $\gamma_1,\gamma_2,\gamma_3$ be generators for the base group of $W$ and let $a=(2 3),b=(1 2 3)$ be generators for $Sym(3)$
acting coordinate-wise on the base group. It follows that $\gamma:=\gamma_1\gamma_2\gamma_3$ commutes with $a$ and $b$ and thus lies in the
centre of $W$. Let $H= \langle \gamma \rangle$, so $\mu(H)=4$. \vspace{12pt}

Set $x=\gamma_1^{-1}\gamma_2^{2}\gamma_3^{-1}$ and $y=\gamma_1^{-1}\gamma_2^{-1}\gamma_3^{2}$. Then it readily follows that$$x^a=x^b=y, \ y^a=x,
\ y^b=x^{-1}y^{-1},$$ so that $G=\langle x,y,a,b \rangle$ is isomorphic to $G(4,4,3)$. Moreover with a little calculation, it can be shown that
$G \cap H =\{1\}$. \vspace{12pt}

It now follows that $W$ is an internal direct product of $G$ and $H$. Therefore by Theorem \ref{theorem:G443}, we have $$12=\mu(G \times H) <
\mu(G) + \mu(H)=16$$ and so $G$ and $H$ form a counter-example to (\ref{eq:directsum}) of degree 12. \vspace{12pt}

Finally, we remark that using the result from \cite{S07} that $\mu(G(p,p,p))=p^2$ for $p$ a prime, it follows that $\mu(G(3,3,3))=9$. However
the centralizer, $C_{Sym(9)}(G(3,3,3))$ in $Sym(9)$ is a proper subgroup of $G(3,3,3)$. So it is not possible to get a counter-example to
(\ref{eq:directsum}) of degree 9 in this case, by this method. \vspace{12pt}

Similarly by realizing $G(2,2,3)$ as $Sym(4)$, it is immediate that $\mu(G(2,2,3))=4$ and again a counter-example to (\ref{eq:directsum}) of
degree 4 is impossible by this method.

\bibliographystyle{plain}

\end{document}